\documentclass{amsart}

\usepackage{amssymb}
\usepackage{amsmath}
\usepackage{amsthm}
\usepackage{verbatim}
\usepackage{diagbox}
\usepackage{thmtools}
\usepackage{hyperref}
\usepackage[shortlabels]{enumitem}
\usepackage{todonotes}
\usepackage{orcidlink}

\usepackage{tikz}
\usetikzlibrary{graphs}

\usepackage[backend=bibtex,maxbibnames=99]{biblatex}
\newcommand{\ctb}{{\aleph_0}}

\let\epsilon\varepsilon
\newcommand{\vp}{\varphi}
\let\phi\varphi

\newcommand{\mc}{\mathcal}
\newcommand{\sbeq}{\subseteq}

\newcommand{\sat}{\vDash}
\newcommand{\bsl}{\backslash}

\newcommand{\set}[1]{\left\{#1\right\}}

\newcommand{\ang}[1]{\left\langle#1\right\rangle}
\newcommand{\restr}{\restriction}

\let\bar\overline 
\let\0\varnothing 
\let\emptyset\varnothing 
\let\tilde\widetilde 

\DeclareMathOperator{\Mod}{Mod}

\DeclareMathOperator*{\bigdoublevee}{\bigvee\mkern-15mu\bigvee}

\newcommand{\Pibz}{\mathbf{\Pi}^0}
\newcommand{\Sibz}{\mathbf{\Sigma}^0}

\theoremstyle{plain}
\newtheorem{thm}{Theorem}[section]
\newtheorem{cor}[thm]{Corollary}
\newtheorem{prop}[thm]{Proposition}

\newtheorem*{cor:pseudofin}{\autoref{pseudofin}}
\newtheorem*{thm:bddUnbdd}{\autoref{bddUnbdd}}

\theoremstyle{definition}
\newtheorem{defn}[thm]{Definition}
\newtheorem{notn}[thm]{Notation}
\newtheorem{rmk}[thm]{Remark}
\newtheorem{quest}[thm]{Question}
\newtheorem{obs}[thm]{Observation}

\theoremstyle{remark}

\title{A Complete Bounded Theory with Unbounded Types}
\author[Hongyu Zhu]{Hongyu Zhu \orcidlink{0000-0002-2522-6319}}
\address{Department of Mathematics\\ University of Wisconsin--Madison \\
  480 Lincoln Drive, Madison, Wisconsin 53706-1325\\ USA}
\email{\href{mailto:hongyu@math.wisc.edu}{hongyu@math.wisc.edu}}
\urladdr{\url{https://people.math.wisc.edu/~hzhu322/}}
\date{\today}

\begin{document}
\begin{abstract}
  One measure of the complexity of a first-order theory, and similarly a type, is the complexity of the formulas required to axiom\-atize it. We say a theory is bounded if there is an axiom\-atization involving only $\forall_n$-formulas for some finite $n$, and unbounded otherwise. One might expect bounded theories to have only bounded types. In fact, an analogue holds in infinitary logic, where the complexity of a Scott sentence roughly agrees with the complexity of the most complicated automorphism orbit. Our main result, however, shows this is not the case in the first-order setting: Namely, there can be a bounded theory, in fact $\forall_1$-axiom\-atizable, which has unbounded types. 
\end{abstract}
\maketitle

\section{Introduction}
There have been many ways to measure the complexity of a theory. To begin with, finite axiom\-atizability has been studied since the very early days of logic: For example, the contrasting results that ZFC is not finitely axiom\-atizable \cite{Montague1964-MONSCA-7} and that NBG (von Neumann–Bernays–G\"odel set theory) is \cite{Godel1940-GDETCO} remain one of the greatest distinctions between the two versions of set theory. One weakening of this notion is recursive axiom\-atizability, which is about the computational complexity of the set of formulas in the axiom\-atization and is featured in G\"odel's Incompleteness Theorems \cite{godel_uber_1931}. The notion of complexity we will use is another weakening of finite axiom\-atizability, looking at the syntactical complexity of individual formulas instead:

\begin{defn}
Say a theory is \emph{boundedly axiom\-atizable} (or \emph{bounded} for short) if it is $\forall_n$-axiom\-atizable (see \autoref{forallFml}) for some finite $n$. Otherwise, it is not boundedly axiom\-atizable (or \emph{unbounded} for short).
\end{defn}

There has been recent interest in this notion. For example, Enayat and Visser \cite{EV} showed the incompleteness of any consistent bounded sequential theory in a finite language. In addition, it relates nicely with descriptive complexity:  Andrews, Gonzalez, Lempp, Rossegger, and Zhu \cite{AGLRZ} showed that given a theory $T$, its set of (countable) models (see \autoref{modT}) is $\Pibz_n$ if and only if $T$ is $\forall_n$-axiom\-atizable. Therefore, a theory is bounded if and only if its set of models is $\Pibz_n$ for some finite $n$. In addition, the authors showed in the same paper that if $T$ is complete (hence also if $T$ is a type), then it is unbounded if and only if its set of models is $\Pibz_\omega$-complete.

The main goal of this paper is to investigate the relationship between the boundedness of a theory and that of the types of the theory. Observe that unbounded theories always have unbounded types, since the theory is the type of the empty tuple. Therefore, one could conjecture that the converse holds as well, that bounded theories will only have bounded types. In a sense, this holds in infinitary logic: There, the analogue of a complete theory is a Scott sentence, which completely characterizes a countable structure; And the analogue of a (complete) type is an infinitary definition (without parameters) of an automorphism orbit. Montalb\'an \cite{montalban_robuster_2015} shows that for any structure $\mc A$, having a $\Pi^{\mathrm{in}}_{\alpha+1}$ Scott sentence is equivalent to all automorphism orbits being $\Sigma^{\mathrm{in}}_{\alpha}$-definable. Nevertheless, our main result shows that this is not the case in the first-order setting:

\begin{thm:bddUnbdd}
There is a complete theory $T$ which is bounded (in fact $\forall_1$-axiom\-atizable) but has unbounded types. In addition, it is strictly superstable.
\end{thm:bddUnbdd}

One major difficulty in proving this theorem is obtaining an unbounded type. Usually, such types show up either when the theory is already complicated to begin with (e.g. true arithmetic), or as Marker extensions which simultaneously increase the complexity of the theory and the types. In both cases, the underlying theories are unbounded, so we needed new machinery.

Notably, such a theory is far from being model complete, despite being $\forall_2$-axiom\-atizable: Any unbounded type must contain, for any $n$, formulas not equivalent to any $\forall_n$-formula (otherwise there would be a $\forall_n$-axiom\-atization of the type). However, every formula in a model complete theory is equivalent to both a $\forall_1$-formula and an $\exists_1$-formula. So in a sense, our result witnesses strongly the failure of the converse of the well-known theorem that every model-complete theory is $\forall_2$-axiom\-atizable (see for example \cite{Hodges}).

The existence of such a theory also connects to the $\omega$-Vaught's Conjecture, a structural strengthening of Vaught's Conjecture, introduced by Gonzalez and Montalb\'an \cite{gonzalez_-vaughts_2023}. There, the authors introduce the notion of \emph{Vaught ordinal} for a theory, which quantifies the level at which the ``countable-or-continuum'' behavior occurs in the models of the given theory. The $\omega$-Vaught's Conjecture asserts the existence of such an ordinal below $\omega$ for all (infinitary) theories, while the original Vaught's Conjecture is equivalent to the existence of it below $\omega_1$, which could then be called the $\omega_1$-Vaught's Conjecture. In an upcoming paper \cite{typeWVC}, the author will prove a slight weakening of $\omega$-Vaught's Conjecture for $\omega$-stable theories, namely the $(\omega\cdot2)$-Vaught's Conjecture. Due to the use of types there, if a bounded $\omega$-stable theory can have only bounded types, the result would then be improved to the full $\omega$-Vaught's Conjecture for $\omega$-stable theories. In light of such discussions, our main result suggests that such an improvement may be too much to hope for.

The remainder of the paper is devoted to constructing such a theory $T$. The main idea used is to introduce complexity in types by allowing for arbitrarily complicated trees in the theory, and then obfuscate it by adding all finite trees so that a sentence without parameters cannot have access to the complexity. A key motivation of the entire construction is the following theorem:

\begin{cor:pseudofin}
Finite-height trees (in the language $\set{\operatorname{Pred}}$ with only the predecessor relation or, equivalently, in the language $\set{\operatorname{\le}}$) are pseudofinite.
\end{cor:pseudofin}

\section{Background}
First, we make precise what is meant by $\forall_n$- and $\exists_n$-formulas.

\begin{defn}\label{forallFml}
\phantom{p}

  \begin{itemize}
    \item $\forall_0=\exists_0$ is the set of quantifier-free formulas.
    \item The set of \emph{$\forall_{n+1}$-formulas} consists of all formulas of the form $\forall x_1\ldots\forall x_n\ \vp(x_1,\ldots,x_n,y_1,\ldots,y_m)$ where $\vp(x_1,\ldots,x_n,y_1,\ldots,y_m)$ is an $\exists_n$-formula.
    \item The set of \emph{$\exists_{n+1}$-formulas} consists of all formulas of the form $\exists x_1\ldots\exists x_n\vp(x_1,\ldots,x_n,y_1,\ldots,y_m)$ where $\vp(x_1,\ldots,x_n,y_1,\ldots,y_m)$ is a $\forall_n$-formula.
    \item If a theory has an axiom\-atization consisting entirely of $\forall_n$-formulas, then we say it is \emph{$\forall_n$-axiom\-atizable}.
  \end{itemize}
\end{defn}

Next, we adopt standard conventions in computable structure theory that the language $\mc L$ is always countable, the countable structures always have domain $\omega$, and that the countable models of a theory are viewed as a subset in Cantor space:

\begin{notn}\label{modT}
Fix a listing $\set{\vp_i\mid i\in\omega}$ of all atomic $(\mc L\cup\omega)$-sentences, where elements of $\omega$ are viewed as constants. We identify any $\mc L$-structure $\mc A$ with domain $\omega$ with its atomic diagram $\mc D(\mc A)\in2^\omega$, i.e. $\mc D(\mc A)(i)=1$ if $\mc A\sat\vp_i$ and 0 otherwise.

The set $\Mod(T)$ of an $\mc L$-theory $T$ is the set of all countable models of $T$.

\end{notn}

\begin{defn}
By the \emph{descriptive complexity} of a theory $T$ we mean the descriptive complexity of $\Mod(T)\sbeq 2^\omega$. For example, we say $T$ is $\Pibz_n$ if $\Mod(T)$ is a $\Pibz_n$ subset of $2^\omega$. The same applies to a type (or a partial type) $p(\bar x)$ in language $\mc L$, by viewing $p(\bar x)$ as the theory $p(\bar a)$ in language $\mc L\cup\set{\bar a}$ where $\bar a$ is a new set of constants (matching the length of $\bar x$).
\end{defn}


\section{Base Theory}

We will start by defining a ``base theory'' $T_0$, on top of which our final theory $T$ will be defined. The models are basically trees, except that we name the levels explicitly and the predecessor relation on each level is considered a separate relation. Then we prove a few basic properties of $T_0$.

\begin{defn}
Let $\mc L_0=\set{P_i\mid i\in\omega}\cup\set{<_i\mid i\in\omega}$, where $P_i$ are unary predicates and $<_i$ are binary relations.
\end{defn}

\begin{notn}
Let $P$ denote $\bigdoublevee_{i\in\omega}P_i$, and $<$ denote $\bigdoublevee_{i\in\omega}<_i$. These are not in general definable in our theories, but will be used as shorthands in our arguments.
\end{notn}

\begin{rmk}
The intended interpretations of the language and the base theory $T_0$, to be defined below, are as follows: If an $\mc L_0$-structure $\mc M$ is a model of $T_0$, then $<^{\mc M}$ defines a forest (disjoint union of trees) on domain $P^{\mc M}$. $P_i^{\mc M}$ are all the $i$-th level nodes, with the 0-th level being all the root nodes. $<_i^{\mc M}$ is the restriction of $<^{\mc M}$ to $P_i^{\mc M}\times P_{i+1}^{\mc M}$.
\end{rmk}

\begin{defn}\label{baseTh}
Let $T_0$ be the $\mc L_0$-theory that says (for every $i,j\in\omega$ with $i\ne j$):

\begin{itemize}
  \item $P_i\cap P_j=\emptyset$;
  \item $<_i\ \sbeq P_i\times P_{i+1}$;
  \item (Existence of predecessors) $\forall x\in P_{i+1}\,\exists y\in P_i\,(y<_ix)$;
  \item (Uniqueness of predecessors) $\forall x\forall x'\forall y\,((x<_iy\wedge x'<_iy)\to x=x')$;
\end{itemize}
\end{defn}

Notice that $\mc L_0$ is countable and relational, and $T_0$ is $\forall_2$-axiom\-atizable. 

\begin{obs}\label{ctbmdls}
$T_0$ has $\ctb$ finite models (up to isomorphism). In fact, it has finitely many models of each finite cardinality $n$.
\end{obs}

\begin{obs}
Any (not necessarily finite) disjoint union of models of $T_0$ remains a model of $T_0$.
\end{obs}

\section{Constructing the Theory}
Now that we have a base theory $T_0$, we can start defining the theory $T$ satisfying the conclusion of \autoref{bddUnbdd}. The idea is as follows: We want a complete theory, so we will make models of $T$ exhibit some ``generic'' behavior by requiring all finite models of $T_0$ to be present (infinitely often). By adding new constants to represent those finite models, we preserve the low complexity of the theory itself, while also leaving room for types to behave wildly.

\emph{Construction:} Since $T_0$ has $\ctb$ finite models, list all of them as $\set{\mc M_i}_{i<\omega}$. For each $j<\omega$, let $C_i^j$ be a set of new constants of size $\left|M_i\right|$ (hence finite). Let $C=\bigcup_{i,j\in\omega}C_i^j$.

\begin{defn}\label{theoryDef}
Let $\tilde{\mc L}=\mc L_0\cup C$. Let $T$ be the $\tilde{\mc L}$-theory that says:
\begin{itemize}
  \item $T_0$;
  \item The constants in $C$ are pairwise distinct;
  \item The $\mc L_0$-substructure with domain $C_i^j$ is isomorphic to $\mc M_i$; 
  \item For each $k\in\omega$: For all $x,y$, if $x\in C_i^j$ and $x<_ky\vee y<_k x$, then $y\in C_i^j$.
\end{itemize}
\end{defn}

Notice that $\tilde{\mc L}$ is countable, $T$ is $\forall_2$-axiom\-atizable, and $T$ has only infinite models.

Next, we consider the ``minimal'' model of $T$, one that has nothing other than the constants. (It will end up being the prime model.)

\begin{prop}\label{minmdl}
There is a unique model $\mc C\sat T$ whose domain is equal to $C^{\mc C}$, i.e. every element is (the interpretation of) a constant. In addition, $\mc C$ embeds into every model of $T$.
\end{prop}

\begin{proof}
Straightforward from the axioms, where we completely specified the isomorphism types of the trees the constants are on.
\end{proof}

From now on, we refer to the substructure of all the constants of any $\mc M\sat T$ as the copy of $\mc C$ inside $\mc M$.

\begin{defn}
Let $\mc C_i^j$ be the $\mc L_0$-substructure of $\mc C$ with domain $C_i^j$.
\end{defn}

\begin{notn}
By $\mc M^\omega$ we mean the countable disjoint union of the structure $\mc M$ (in a relational language).
\end{notn}

\begin{obs}\label{cdisju}
As $\mc L_0$-structures, $\mc C\cong\bigsqcup_i\mc M_i^\omega.$
\end{obs}

\begin{obs}\label{mdec}
For all $\mc M\sat T$, $\mc M\cong\mc C\sqcup(\mc M\bsl\mc C)$ over $\mc L_0$ (where $\mc M\bsl\mc C$ is the $\mc L_0$-substructure of $\mc M$ with domain $M\bsl C$).
\end{obs}

\section{Completeness of the Theory}

We move on to show the completeness of $T$. The main tool we will use is the Ehrenfeucht-Fra\"iss\'e game $EF^{\mc L}_k(\mc M,\mc N)$\footnote{We add a superscript $\mc L$ here to make explicit that $\mc M, \mc N$ are considered as $\mc L$-structures here.} (see \cite[Section 3.2]{Hodges}), which help characterize the theory (and types thereof) with invariants we introduce. The perfect-information game $EF^{\mc L}_k(\mc M_0,\mc M_1)$ goes on for $k$ steps, and in each step $\forall$ chooses an element from $M_0$ or $M_1$, followed by $\exists$ choosing an element from the other structure (trying to ``match'' $\forall$'s choice). In the end, collect all chosen elements $\bar m_i$ from $M_i$, and $\exists$ wins the play if and only if there is an isomorphism $f:\ang{\bar m_0}_{\mc M_0}\to\ang{\bar m_1}_{\mc M_1}$ identifying $\forall$'s choice at each step with $\exists$'s corresponding choice.

\begin{notn}
For two $\mc L$-structures $\mc M,\mc N$, write $\mc M\equiv_{\mc L,k}^{EF}\mc N$ if $\exists$ has a winning strategy in $EF^{\mc L}_k(\mc M,\mc N)$. And say $\mc M\equiv^{EF}_{\mc L}\mc N$ if $\mc M\equiv_{\mc L, k}^{EF}\mc N$ for every finite $k$, i.e. $\exists$ wins every $EF^{\mc L}_k(\mc M,\mc N)$ of finite length.
\end{notn}

The main reason we use $EF^{\mc L}_k(\mc M_0,\mc M_1)$ is the theorem below, which follows from \cite[Corollary 3.3.3]{Hodges}.

\begin{thm}\label{EF}
If $\mc M\equiv^{EF}_{\mc L}\mc N$ in a finite language $\mc L$, then $\mc M\equiv\mc N$ in the same language. As a result, if $\mc M\equiv^{EF}_{\mc L'}\mc N$ in every finite $\mc L'\sbeq\mc L$, then $\mc M\equiv\mc N$ over $\mc L$.
\end{thm}

\subsection{The Invariants}
Going to finite sublanguages allows us to work with finite-height trees, where we can find invariants to use in the EF-games. We first define the finite sublanguage of $\mc L_0$ that describes trees up to a finite height $h$:
\begin{defn}
For $0<h<\omega$, let $\mc L_h\sbeq\mc L_0$ be the finite sublanguage consisting of $P_i$ for $i\le h$, and $<_i$ for $i<h$. 
\end{defn}
Notice that $\mc L_0=\bigcup_{0<h<\omega}\mc L_h$, and $\mc L_h$-reducts of models of $T_0$ are forests of height at most $h$. In what follows we will be frequently using tree terminologies on models $\mc M\sat T_0$ (or their reducts), where the tree structure is understood to be $(P^{\mc M},<^{\mc M})$.

\begin{defn}
By an \emph{$h$-forest} we mean an $\mc L_h$-reduct of a model of $T_0$. We may also say \emph{$h$-tree} when it has a single root. Its domain (as a forest) $P_{\le h}$ is defined as $\bigcup_{i\le h}P_i.$
\end{defn}

\begin{defn}
In a forest, the \emph{level} of $x\in P$ is the the unique $l$ such that $P_l(x)$.
\end{defn}

To show the completeness of $T$, one attempt is to show any $\mc M\sat T$ is elementarily equivalent to $\mc C$ in any finite sublanguage of the form $\mc L'=\mc L_h\cup C_0$ where $C_0\sbeq C$ is finite. The major problem is that $\mc C$ is a forest of finite trees while $\mc M$ may contain infinite trees. But when viewed from $\mc L'$, every tree will have finite height, and in this case we will show that infinite trees can be sufficiently well approximated by finite ones. For such purposes we introduce the following invariant, which summarizes the information of the tree above a node inductively using that of its children.

Fix $0<h<\omega, k\in\omega.$
\begin{defn}\label{clr}
The \emph{$k$-bounded coloring} of an $h$-forest $\mc M$ is the function $\Lambda$ defined on $M$ as follows: given $x\in M$, $\Lambda(x)=\ang{\Lambda_l(x),\Lambda_{\sigma}(x)}$ where
\begin{itemize}
  \item $\Lambda_l(x)$ is the level of $x$ (or $-1$ if not in $P_{\le h}$).
  \item $\Lambda_{\sigma}(x)$ is a set of pairs $\ang{\lambda_i,n_i}$. It is defined inductively, from the leaves down to the root, as follows:
      \begin{itemize}
        \item Let $\Lambda_{\sigma}(x)$ be the set of all pairs $\ang{\lambda,n}$, where $\lambda$ is the \emph{color} of a successor $y$ of $x$ (i.e. $\lambda=\Lambda(y)$), and $n=\min(k,m)$ where $m$ is the number of successors of $x$ with color $\lambda$.
        \item In particular, $x\in P_{\le h}$ is a leaf if and only if $\Lambda_{\sigma}(x)=\emptyset$.
      \end{itemize}
\end{itemize}
\end{defn}
Intuitively, $\emptyset$ is the color of leaves, each $\lambda_i$ is a (previously defined) color, and $n_i$ is the number of successors of $x$ with color $\lambda_i$ (but capped at $k$).

\begin{rmk}
By induction (noting that we are working with trees of a fixed finite height $h$), it's clear that there are only finitely many possible colors, and thus the range of $\Lambda$ (for all $\mc M$) lives in a finite set (depending only on $h$ and $k$).
\end{rmk}
\begin{defn}
For fixed $h,k$, a \emph{$(k,h)$-color} will mean one (among finitely many) that could be the color of some element in the $k$-bounded coloring of some $h$-forest.

For any element $x$ (in a model of $T$), its \emph{$(k,h)$-color} is $\Lambda(x)$ where $\Lambda$ is the $k$-bounded coloring of its ambient model as an $h$-forest.

When $k,h$ are understood from context, we will say color for $(k,h)$-color.
\end{defn}

\begin{prop}\label{tildeprop}
For every $(k,h)$-color $\lambda$, there is a finite $h$-tree $\mc Y_\lambda$ whose root has color $\lambda$.
\end{prop}

\begin{proof}
  From the definition of our coloring, we can build inductively a $(\le k)$-branching $h$-tree $\mc Y_\lambda$ with root color $\lambda$. But such trees must be finite.
\end{proof}

\subsection{Winning Strategies for $\exists$}

Now we show the colors describe an $h$-forest $\mc M$ sufficiently well, in the following sense.

\begin{thm}\label{clrComp}
  Fix $0<h<\omega,n\in\omega$. Let $k=n(h+1)$. Suppose that:

  \begin{itemize}
    \item $\mc M_0,\mc M_1$ are $h$-forests such that for each $(k,h)$ color $c$, $\mc M_0$ and $\mc M_1$ have the same number of roots with color $c$.
    \item $\bar a_0\in\mc M_0, \bar a_1\in\mc M_1$ are tuples of the same length and are closed under predecessor;
    \item The map $\bar a_0\mapsto\bar a_1$ is a $(k,h)$-color-preserving isomorphism.
  \end{itemize}

  Then $(\mc M_0,\bar a_0)\equiv^{EF}_{\mc L_h,n}(\mc M_1,\bar a_1)$.
\end{thm}

\begin{proof}
  First, we modify the EF-game by extending the length from $n$ to $k=n(h+1)$, but requiring that $\forall$ cannot choose an element unless its predecessor has been chosen before in the game (or the element has no predecessor). Winning this prolonged version of the game suffices because: whenever $\forall$ plays an element $x$ in the original length-$n$ game, $\exists$ can act as if $\forall$ played all of the (at most $h+1$) ancestors of $x$ in order (starting from the root) in the prolonged game, and winning the latter game clearly tells $\exists$ how to win the former. Since each step in the original game takes up at most $h+1$ steps in the prolonged game and $k=n(h+1)$, we have ensured that $\exists$ can win before using up all moves in the prolonged version.

  Now we work in the prolonged game with length $k$, and suppose $\forall$ can only choose an element after its predecessor has shown up. This means that at every step, $\forall$ only makes one of the following choices:
\begin{itemize}
  \item A previously chosen element (by $\forall$ or $\exists$);
  \item An element of $\bar a_0\cup\bar a_1$;
  \item An element outside of $P_{\le h}$ (from either structure);
  \item An element of $P_0$ (i.e. the root of a tree) (from either structure);
  \item Or, an successor of a previously chosen element.
\end{itemize}
Correspondingly, we describe a winning strategy for $\exists$ (which builds a $(k,h)$-color-preserving partial isomorphism; since the language is relational, the substructure generated by any tuple is just itself, so the partial isomorphism is totally determined by $\exists$'s choices). At the same time, we verify inductively that at each step: (1) the element $\exists$ chose has the same color as the one $\forall$ chose at the same step (when this is not obvious); and (2) $\exists$ successfully builds a partial isomorphism up to that step. Note that our assumption of $\bar a_0\mapsto\bar a_1$ being a $(k,h)$-color preserving isomorphism covers the base case.

  \begin{itemize}
    \item If $\forall$ chooses a previously chosen element (or one in $\bar a_0\cup\bar a_1$), then $\exists$ makes the same choice as it did earlier.
    \item If $\forall$ chooses an element of $\bar a_0\cup\bar a_1$, then $\exists$ chooses the corresponding element according to the (by our third assumption) $(k,h)$-color-preserving isomorphism $\bar a_0\mapsto\bar a_1$.
    \item If $\forall$ chooses an element outside $P_{\le h}$, then $\exists$ chooses an element outside $P_{\le h}$ in the other structure.
    \item If $\forall$ chooses a new root node $x$, then $\exists$ chooses a root $y$ in the other structure with the same $(k,h)$-color. This is always possible by our first assumption.
    \item If $\forall$ chooses a new successor $x$ of a previously chosen element $x_0$ (that has not been chosen previously): suppose $\exists$ responded to $x_0$ with $y_0$ previously. By induction hypothesis, $x_0$ and $y_0$ have the same color. Now choose $y$ to be a new successor of $y_0$ having the same color $c$ as $x$ that have not been chosen before. The existence of such an element is verified as follows: Say $x,x_0\in\mc M_i$ with $i<2$. (1) If $x_0$ has at least $k$ successors of color $c$ in $\mc{M}_i$: then $y_0$ has at least $k$ successors of color $c$ in $\mc{M}_{1-i}$ (by definition of the color on $x_0,y_0$ and that they have the same color). But at this stage we have chosen fewer than $k$ elements in each structure as the game has length $k$, so there is a new one available to $\exists$. (2) If $x_0$ has at most $k$ successors of color $c$ in $\mc{M}_i$: then similarly $y_0$ has the same number of successors of color $c$ in $\mc M_{1-i}$. By induction hypothesis, we built a partial $(k,h)$-color-preserving isomorphism previously, so the number of successors with color $c$ that have been chosen in both structures are the same. Since $\forall$ can find a new element on one side, $\exists$ must be able to do the same on the other side, so we are done.
  \end{itemize}
\end{proof}

\subsection{Finite-height Trees}

A first consequence of \autoref{clrComp} is that for every $h$-forest $\mc M$ and every $n$, there exists a $k$ with $\mc M\equiv_{\mc L_h,n}^{EF}\tilde{\mc M}_{k,h}$. This allows us to show that finite-height trees are pseudofinite.

\begin{prop}\label{pruneEF}
Fix $0<h<\omega,n\in\omega$ and let $k=n(h+1)$. For every $h$-tree $\mc M$, we have $\mc M\equiv_{\mc L_h,n}^{EF}{\mc Y}_{\lambda}$ (from \autoref{tildeprop}) where $\lambda$ is the $(k,h)$-color of the root of $\mc M$.
\end{prop}

\begin{proof}
  Apply \autoref{clrComp} to $\mc M$ and ${\mc Y}_{\lambda}$ (with $a_0=a_1=\emptyset$), noting that the corresponding roots have the same color.
\end{proof}

\begin{thm}\label{pseudofin}
Finite-height trees (in the language $\set{\operatorname{Pred}}$ with only the predecessor relation or, equivalently, in the language $\set{\operatorname{\le}}$) are pseudofinite.
\end{thm}
\begin{proof}
First we work in the language $\set{\operatorname{Pred}}$. Fix a tree $\mc M$ with height $h$ and a formula $\vp$ with $\mc M\sat\vp$. Then $\mc M$ is definitionally equivalent to an $h$-tree $\mc M'$ (with all elements in $P_{\le h}$), by interpreting $\operatorname{Pred}$ as $\cup_{i<h}<_i$ in one direction; and interpreting $P_i$ as the elements on the $i$-th level (expressible with $\operatorname{Pred}$), $<_i$ as $\operatorname{Pred}\restr(P_i\times P_{i+1})$ in the other direction. Let $\vp'$ be the corresponding formula of $\vp$ under this interpretation, so that $\mc M'\sat\vp'$. By writing $\vp'$ using game-normal formulas (see \cite[Theorem 3.3.2]{Hodges}), we see that there exists an $n$ such that for all $\mc L_h$ structures $\mc N'$, $\mc M'\equiv^{EF}_n\mc N'$ implies $\mc N'\sat\vp'$. By \autoref{pruneEF}, we can take $\mc N'=\mc Y_{\lambda}$, with $\lambda$ being the $(n(h+1),h)$-color of the root of $\mc M$, to guarantee $\mc N'\sat\vp'$. Let $\mc N$ be the tree obtained from $\mc N'$ using the same definitional equivalence, so $\mc N\sat\vp$. Now by \autoref{tildeprop}, $\mc N'$ is finite, so $\mc N$ is finite as well.

The claim for the language $\set{\le}$ follows from that for $\set{\operatorname{Pred}}$, since we again have a definitional equivalence given a tree of fixed finite height $n$.
\end{proof}

\subsection{Completeness}
The next application of \autoref{clrComp} is to show the completeness of $T$.

\begin{cor}\label{tpDesc}
  Fix $0<h<\omega,n\in\omega$. Let $k=n(h+1)$. For $\tilde{\mc L}$-structures $\mc M, \mc N$ both satisfying $T$, suppose $\bar a\in\mc M, \bar b\in\mc N$ are tuples of the same length closed under predecessor with $\bar a\mapsto\bar b$ being a $(k,h)$-color-preserving isomorphism. Then $(\mc M,\bar a)\equiv^{EF}_{\mc L_h,n}(\mc N,\bar b)$.
\end{cor}
\begin{proof}
  Apply \autoref{clrComp} to $(\mc M,\bar a)$ and $(\mc N,\bar b)$, The only assumption not given directly is that $\mc M, \mc N$ have the same number of roots with a given color $\lambda$, which holds because: Recall from \autoref{tildeprop} that the color of any root in any $h$-forest is a color of a root of a \emph{finite} $h$-forest. But by our definition of $T$ (which include all the constants in $C$), all such colors appear at least (thus exactly) $\ctb$ times in both $\mc M$ and $\mc N$, since both satisfy $T$. So $\mc M,\mc N$ have the same number of roots with any given color $\lambda$.
\end{proof}

\begin{thm}\label{complete}
$T$ is complete.
\end{thm}

\begin{proof}
  It amounts to proving all $\mc M,\mc N\sat T$ are elementarily equivalent. In turn, by \autoref{EF} it suffices to show $\mc M\equiv_{\mc L',n}^{EF}\mc N$ for every finite sublanguage $\mc L'$ of $\tilde{\mc L}$. We may assume $\mc L'=\mc L_h\cup \bar c$ for some $0<h<\omega$ and some finite $\bar c\sbeq C$ closed under predecessor. Now the requirement becomes: \[(\mc M,\bar c^{\mc M})\equiv_{\mc L_h,n}^{EF}(\mc N,\bar c^{\mc N}).\]

  To this end, we apply \autoref{tpDesc} to $(\mc M,\bar c^{\mc M})$ and $(\mc N,\bar c^{\mc N})$. We check the assumptions: $\bar c^{\mc M}, \bar c^{\mc N}$ clearly have the same length, and are closed under predecessor by definition. In addition, the color and isomorphism type of $\bar c$ (in either structure) are fully specified by the theory $T$, so the corresponding map $\bar c^{\mc M}\mapsto\bar c^{\mc N}$ is a $(k,h)$-color-preserving isomorphism. Hence we are done.
\end{proof}

\section{Analysis of Types}

Before we discuss the existence of complicated types, we need to understand (and in particular count) types over $T$. Fortunately, this follows from our previous analyses: the tree colorings (as in \autoref{clr}) play a role similar to an elimination set for the theory.

\begin{prop}
For any $0<h<\omega, k\in\omega$ and any $(k,h)$-color $c$, there is a first-order formula in free variable $x$ expressing the fact that the $(k,h)$-color of $x$ is $c$.
\end{prop}
\begin{proof}
  By inducting on the definition of the coloring.
\end{proof}

The description of types below follows directly from \autoref{tpDesc}.

\begin{prop}
Any $n$-type $p(\bar x)$ over $T$ is completely determined by: (1) Whether or not each $x_i$ is a constant; (2) For each $i$, the unique $n$ such that $x_i\in P_n$ (or the nonexistence thereof); (3) For each $k,h$, the $(k,h)$-color of each $x_i$ and ancestors; (4) For any $x_i,x_j\in\bar x$, the unique $(u,v)$ such that the $u$-th predecessor of $x_i$ is equal to the $v$-th predecessor of $x_j$ (or the nonexistence of $(u,v)$, i.e. they are on different trees).

\end{prop}
This characterization extends to types over a set $S$: only item (4) needs to take $S$ into additional consideration.

\begin{cor}\label{superstable}
$T$ is strictly superstable.
\end{cor}
\begin{proof}
By our characterization of types above, there can be at most $\kappa+2^\ctb$ $1$-types $p(x)$ over a set $S$ of cardinality $\kappa$, so $T$ is superstable. (Note that item (4) is determined by the highest-level element in $S$ that shares the same tree with $x$.) On the other hand, for each $X\in 2^\omega$, consider the partial type $p_X(x)$ that says $x$ is the root of a tree which has a leaf on level $n$ if and only if $n\in X$. Clearly these are continuum many partial 1-types over $\varnothing$ and are pairwise incompatible, so $T$ is not $\omega$-stable.
\end{proof}

\section{Complicated Types}
Now we start working on complicated (complete) types. The idea is that the theory allows us to put complicated trees in the model, and having access to the root reveals the complexity. To do this more formally, we define the following stronger notion of continuous reducibility.

\begin{defn}
For two structures $\mc M\not\cong\mc N$ (in the same language), say $(\mathbf\Sigma^0_k,\mathbf\Pi^0_k)\le_c^*(\mc M,\mc N)$ if for any $X$ and any $\Sigma^0_k(X)$ set $A\sbeq 2^\omega$, there exists an $X$-computable reduction $f:2^\omega \to 2^\omega$ such that $x\in A\iff f(x)\cong\mc M$ and $x\notin A\iff f(x)\cong\mc N$, and the $X$-code for $f$ is uniformly computable from a $\Sigma^0_k(X)$-code for $A$.
\end{defn}

\begin{rmk}
If $(\mathbf\Sigma^0_k,\mathbf\Pi^0_k)\le_c^*(\mc M,\mc N)$, then $\set{X\in2^\omega\mid X\cong\mc M}$ is $\Sibz_k$-hard and $\set{X\in2^\omega\mid X\cong\mc N}$ is $\Pibz_k$-hard.
\end{rmk}

To build an unbounded type, we first construct, for every $k$, a tree $\mc M_k$ whose root satisfies a $\mathbf\Pi^0_k$-hard 1-type. Then we combine all of them into a single tree, starting with a tree with countably many nodes $G_k$ definable over the root and putting $\mc M_k$ above each $G_k$. To make sure $\Pi^0_k$-hardness transfers from $\mc M_k$ to the combined tree, we need formulas whose truth values depend only on trees above the free variables. This leads us to the ``local formulas'' defined below:

\begin{defn}
The collection of \emph{local formulas} is the smallest collection of $\mc L_0$-formulas that:
\begin{itemize}
  \item contains all atomic and negated atomic formulas; and
  \item is closed under \emph{local quantifications}, namely quantifiers of the form ${\forall y>_kx}$, ${\exists y>_kx}$, for any variable $x$ and $k\in\omega$.
\end{itemize}
\end{defn}

\begin{rmk}
As is common practice, we will also say a formula is local if it is logically equivalent to a local formula. Under this convention, the set of local formulas is closed under negation, conjunction, disjunction, and local quantification. 
\end{rmk}

When we combine the $\mathbf\Pi^0_k$-hard trees into a single tree, each constituent's root will never be on level 0. So we need the following definition of ``upshift'' for adapting local formulas to when the root is not necessarily on level 0.

\begin{defn}
\phantom{p}

\begin{itemize}
  \item If $\vp$ is a local formula and $k\in\omega$, then let $\vp^{+k}$, the \emph{$k$-th upshift of $\vp$}, be the formula obtained from $\vp$ by replacing every $P_i$ by $P_{i+k}$, and $<_i$ by $<_{i+k}$.

      Clearly, any upshift of a local formula remains local.
  \item If $\mc M\sat T_0$ with $m\in P^{\mc M}$, the \emph{tree above $m$}, denoted as $\mc M_{\ge m}$, is the $\mc L_0$-structure obtained by setting $m$ as the only root node and copying everything above $m$. More formally: say $m\in P_k^{\mc M}$.
      \begin{itemize}
        \item $M_{\ge m}=\set{x\in M\mid x>^{\mc M}m\vee x=m}$.
        \item For $x\in M_{\ge m}, x\in P_i^{\mc M_{\ge m}}\iff x\in P_{i+k}^{\mc M}$.
        \item For $x,y\in M_{\ge m}, x<^{\mc M_{\ge m}}_iy\iff x<_{i+k}^{\mc M}y$.
      \end{itemize}
  It follows that $\mc M_{\ge m}\sat T_0$ with its isomorphism type depending only on that of $(\mc M,m)$.
\end{itemize}
\end{defn}

Now we show that truth of local formulas is preserved ``locally,'' i.e. is determined by the tree above the free variables. The proposition generalizes to more than one variable, but the single-variable version suffices for us.

\begin{prop}
Suppose $\vp(x)$ is a local formula in one variable $x$, $\mc M_i$ are $\mc L_0$-structures, $m_i\in P_{k_i}^{\mc M_i}$ for $k_i\in\omega, i<2$. Suppose in addition that $\mc M_{0,\ge m_0}\cong \mc M_{1,\ge m_1}$, i.e. the trees above the two chosen elements are isomorphic. Then \[\mc M_0\sat\vp^{+k_0}(m_0)\iff\mc M_1\sat\vp^{+k_1}(m_1).\]
\end{prop}
\begin{proof}
  Note that the isomorphism $\mc M_{0,\ge m_0}\cong \mc M_{1,\ge m_1}$ has to send $m_0$ to $m_1$ since it preserves the unique root. Hence, it suffices to show that (for $i<2$)
  \[\mc M_i\sat\vp^{+k_i}(m_i)\iff\mc M_{i,\ge m_i}\sat\vp(m_i).\]
  This follows by first relativizing (in the sense of Theorem 5.1.1 of \cite{Hodges}) $\vp^{+k_i}$ to the tree above $m_i$ in $\mc M_i$, noticing that $\vp^{+k_i}$ is invariant under this relativization (by locality); and then chasing through the definition of the tree above an element.
\end{proof}

Now we begin constructing the $\mathbf\Pi^0_k$-hard trees.

\begin{prop}\label{PairOfTrees}
Uniformly in $k\in\omega, k\ge 1$, we can build computable trees $\mc M_k, \mc N_k\sat T_0$ and a local formula $\vp_k(x)$ such that $(\mathbf\Sigma^0_k,\mathbf\Pi^0_k)\le_c^*(\mc M_k,\mc N_k)$, and $\mc M_k\sat\vp_k(*), \mc N_k\sat\neg\vp_k(*)$ where $*$ is the unique (definable) root node (so $\mc M_k\not\equiv\mc N_k$ witnessed by $\vp_k(*)$, in particular $\mc M_k\not\cong\mc N_k$).
\end{prop}
\begin{proof}
  We take $\mc M_k, \mc N_k$ to be the back-and-forth trees $\mc E_k, \mc A_k$ from \cite[Definition 3.1]{hirschfeldt_realizing_2002}, respectively, except that to work in $\tilde{\mc L}$ we must use $<_n$ to replace the directed edges for appropriate values of $n$. This is possible with all uniformity, since we know the level of each node during the construction: The idea there is to start with a single node as $\mc A_1$ and a single root with infinitely leaf successors as $\mc E_1$; Then inductively, let $\mc A_{k+1}$ be a single root with infinitely many $\mc E_k$'s above, and $\mc E_{k+1}$ be a single root with infinitely many $\mc A_k$'s and $\mc E_k$'s above.

  That $(\mathbf\Sigma^0_k,\mathbf\Pi^0_k)\le_c^*(\mc M_k,\mc N_k)$ uniformly in $k$ follows by relativizing \cite[Proposition 3.2]{hirschfeldt_realizing_2002} to any oracle $X$. Again, the idea there is to build the reductions inductively.

  The formulas $\vp_k(*)$ can be found by adapting \cite[Lemma 3.3(2)]{csima_degrees_2020} to our language $\tilde{\mc L}$, noting the resulting formulas can be made local since the quantifiers appearing in the proof there can be replaced by local ones (uniformly computably). Essentially, these formulas are obtained by writing down the (inductive) definitions of $\mc E_k$ and $\mc A_k$.
\end{proof}

Now we are finally ready to build an unbounded type.
\begin{thm}\label{unbounded}
There is a complete type $p(x)\in S_1(T)$ which is $\Pibz_\omega$-hard. In particular, it is not $\forall_n$-axiom\-atizable for any finite $n$.
\end{thm}
\begin{proof}
  We combine the trees $\mc M_i$ from \autoref{PairOfTrees} into a single tree with root $x$, in a way that the root of each $\mc M_i$ is definable from $x$, and then take the complete type of $x$. (We will see at the end why this suffices.)

  The above can be done, for example, by the following: Let $\mc N\sat T_0$ be the tree defined by: (See Figure \ref{fig:N}.)
  \begin{itemize}
    \item There is a unique root, denoted by $F_0$.
    \item (Having defined all of $F_j$ for $j<i$:) There is a unique $y\in P_{2i}$ which is a leaf and whose only common ancestor with $F_j$ is $F_0$, for all $j<i$. (Call this $y$ $F_i$.)
    \item For $i>0$, there is a unique $z\in P_{2i}$ that is not $F_i$ but has the same predecessor as $F_i$. (Call this $z$ $G_i$.)
    \item The tree above $G_i$ is $\mc M_i$.
  \end{itemize}

\begin{figure}[htbp]
  \centering
  \scalebox{1}[1]{%
  \begin{tikzpicture}[style={sibling distance=35pt,level distance=30pt},grow'=up]
  \node [fill,circle,inner sep=2pt,label=below:{$F_0$}] {}
    child {[fill] circle (2pt)
      child {node [circle, fill, label=below left:{$F_1$}, inner sep=2pt]{}}
      child {node [circle, draw, label=below:{$G_1$}] {$\mc M_1$}}}
    child {[fill] circle (2pt)
      child {[fill] circle (2pt)
        child {[fill] circle (2pt)
          child {node [circle, fill, label=below left:{$F_2$}, inner sep=2pt]{}}
          child {node [circle, draw, label=below:{$G_2$}] {$\mc M_2$}}}}}
    child {node {\ldots}};
\end{tikzpicture}
}%
  \caption{$\mc N$}
  \label{fig:N}
\end{figure}

  Let $\mc M$ be the $\tilde{\mc L}$-structure obtained from $\mc C\sqcup\mc N$ (i.e. the $\mc L_0$-structure is $\mc C\sqcup\mc N$ and the constants are from $\mc C$). Clearly, $\mc M\sat T$. Let $p(x)$ be the 1-type of the root node of the $\mc N$-part in $\mc M$.

  We claim $p(x)$ is $\Pi^0_{\omega}$-hard: Take any $\mathbf\Pi^0_\omega$ set $S$, which we may assume is $\Pi^0_\omega(X)$ for some $X$. Write it as $\bigcap_{0<i<\omega}S_i$ where $S_i$ is uniformly $\Sigma^0_i(X)$. By \autoref{PairOfTrees}, we can uniformly find $X$-computable functionals $f_i$ witnessing $(S_i,\bar S_i)\le_c(\mc M_i,\mc N_i)$. Now our $X$-computable reduction $f$ from $S$ to $\Mod(p(a))$ (where $a$ is a new constant) does the following: Given input $y$, $f(y)$ is the following model:
  \begin{itemize}
    \item Add a copy of $\mc C$ (which can be done computably).
    \item Disjoint from $\mc C$, build a new tree with root $a$ by following the instructions above for building $\mc N$, except that build the tree above $G_i$ using $f_i(y)$ (instead of $\mc M_i$).
  \end{itemize}
  Clearly, if $y\in S$ (i.e. $y$ is in every $S_i$) then $f(y)\cong\mc N$ (because each $f_i(y)$ is actually $\mc M_i$). Otherwise, there is some $i$ such that $f_i(y)$ is $\mc N_i$, in particular satisfies $\neg\vp_i(*)$. Using the definability of $G_i$ from $x$ (in the definition of $\mc N$), we see that $f(y)\sat\neg\vp_i^{+2i}(G_i(a))$, while $\vp_i^{+2i}(G_i(x))\in p(x)$ (where $G_i(x)$ is a formula defining $G_i$ from $x$). Hence $f(y)\not\sat p(a)$, so we are done.
\end{proof}

\begin{rmk}
  $p(x)$ is also unbounded over $T$ since $T$ is itself bounded.
\end{rmk}

Combining everything together, we finally obtain our main result.

\begin{thm}\label{bddUnbdd}
There is a complete theory $T$ which is bounded but has unbounded types. In addition, it is strictly superstable.
\end{thm}
\begin{proof}
  The theory $T$ we constructed has a $\forall_2$ axiom\-atization by definition: see \autoref{theoryDef} and the comments immediately after. It has an unbounded type by \autoref{unbounded}, is complete by \autoref{complete}, and is strictly superstable by \autoref{superstable}.
\end{proof}

\begin{cor}
  The theory above can be taken to be satisfy either of the following:
   \begin{itemize}
     \item $\forall_1$-axiom\-atizable.
     \item $\forall_2$-axiom\-atizable in a relational language.
   \end{itemize}
\end{cor}
\begin{proof}
  For a $\forall_1$ axiom\-atization, we can use function symbols $\operatorname{Pred}_n$ for predecessors to replace $<_n$, reformulating the theory correspondingly.

  For a relational $\forall_2$ axiom\-atization, we can similarly use unary predicates to replace all the constants.

  In both cases we get theories bi-interpretable with the one given in \autoref{theoryDef}, so boundedness and stability properties are preserved.
\end{proof}

\section{Open Questions}
There are several possible ways to strengthen the main result. First, one can examine the stability hierarchy. As mentioned in the introduction section, the existence of a bounded theory with unbounded types prevents one from showing the $\omega$-Vaught's Conjecture for $\omega$-stable theories. However, if such behaviors cannot occur in $\omega$-stable theories, then the proof would go through. 

\begin{quest}
Is there a complete bounded $\omega$-stable theory with unbounded types?
\end{quest}

Second, the model we construct has continuum many countable models (since there are continuum many types over $\emptyset$). In view of problems related to the Vaught's Conjecture, we would like to know if such behaviors occur when there are countably many countable models as well.

\begin{quest}
Is there a complete bounded theory with unbounded types having only countably many countable models?
\end{quest}

In addition, while the theory we have is $\forall_1$-axiom\-atizable, to make it relational we would end up with a $\forall_2$-axiom\-atization, which we do not know is optimal or not for a relational theory. We can rule out $\forall_1$ or $\exists_1$ theories: All complete relational $\forall_1$ theories have the empty set as a model (thus trivial); All complete relational $\exists_1$ theories are too simple as well:

\begin{prop}
  Any complete relational $\exists_1$ theory $T$ can have no relations of arity at least 2 in the language. In particular, they cannot have unbounded types.
\end{prop}
\begin{proof}
  If $R$ is a relation symbol with at least 2 variables, consider the formula $\Phi=\exists x \forall y\ R(x,y,\ldots,y)$. Any $\mc M\sat T$ embeds in a superstructure satisfying $\Phi$ and another satisfying $\neg\Phi$, but both have the same complete theory $T$ (since $T$ is $\exists_1$), a contradiction.

  Hence the language $\mc L$ of such $T$ has only unary relations, so $T$ must be the theory of a structure where any finite combination of the predicates has infinitely many realizations (because any $\mc L$-structure embeds into such a structure). Such theories have QE, thus admit no unbounded types.
\end{proof}

But for $\exists_2$, things remain unclear.
\begin{quest}
Is there a $\exists_2$-axiom\-atizable complete theory with unbounded types?
\end{quest}

Lastly, the theory we have is essentially still a theory of trees, and it is crucial that there are no meaningful relations between sibling nodes. While we attempted something similar for graphs, the lack of this ``level-wise independence'' makes it difficult to proceed. So can we still generalize this construction?

\begin{quest}
Is there a similar construction for other structures, like graphs?
\end{quest}

\section*{Acknowledgments}

The author would like to thank Uri Andrews for his helpful suggestions and comments, and Steffen Lempp for proofreading this paper.

\printbibliography

\end{document}